\documentclass[10pt]{article}
\usepackage{amsmath, amsthm}\usepackage{enumerate}\usepackage{amssymb}\usepackage{bbold}\usepackage{color}
\usepackage{bbm}\usepackage{dsfont}\usepackage{color}\usepackage{xcolor}
\newtheorem{theorem}{Theorem}[section]
\newtheorem{proposition}[theorem]{Proposition}

\newtheorem{question}[theorem]{Question}
\newtheorem{example}[theorem]{Example}

\newtheorem{remark}[theorem]{Remark}

\newtheorem{definition}[theorem]{Definition}

\DeclareMathOperator{\convc}{\xrightarrow[]{\mathbb{c}}}
\DeclareMathOperator{\cc}{\xrightarrow[]{\mathbb{c}}}

\DeclareMathOperator{\co}{\xrightarrow[]{o}}

\DeclareMathOperator{\rc}{\xrightarrow[]{\mathbb{rc}}}
\DeclareMathOperator{\convsc}{\xrightarrow[]{\mathbb{sc}}}

\begin{document}

\title{Rough convergence on Riesz spaces}
\maketitle\author{\centering{{Abdullah Ayd\i n$^{1,*}$, Mehmet Küçükaslan$^{2}$, Mokhwetha Mabula$^{3}$\\ \small $^1$Department of Mathematics, Mu\c{s} Alparslan University, Mu\c{s}, 49250, Türkiye, a.aydin@alparslan.edu.tr\\ \small $^2$ Department of Mathematics, Mersin University, Mersin, Türkiye, mkkaslan@gmail.com\\ \small $^3$ Department of Mathematics and Applied Mathematics, University of Pretoria, Pretoria, South Africa, mabulamd@gmail.com\\ $*$Corresponding Author}

\abstract
{This paper extends the theory of rough convergence from normed linear spaces to the more abstract setting of Riesz spaces. We introduce and systematically develop the concept of rough $\mathbb{c}$-convergence ($rc$-convergence) for nets. A net $(x_\alpha)_{\alpha\in A}$ in a Riesz space $E$ is said to be rough $\mathbb{c}$-convergent to $x\in E$ if there exists a net $(y_\alpha)_{\alpha\in A}$ in $E$ with $y_\alpha \convc \theta$ for a given background convergence $\mathbb{c}$, such that $|x_\alpha-x| \leq y_\alpha + \mathbb{r}$ holds for all $\alpha\in A$, where $\mathbb{r}$ is a fixed positive vector in $E$ representing the roughness degree. The study first establishes that this new construction satisfies the axioms of a formal convergence structure. Key properties of $\mathbb{rc}$-convergence are then investigated, including its relationship with linearity and the continuity of lattice operations. Since the limit of an $\mathbb{rc}$-convergent net is not necessarily unique, the paper dedicates significant analysis to the set of rough $\mathbb{c}$-limit points. Furthermore, a crucial connection is established between the order boundedness of a net and the non-emptiness of its set of $\mathbb{rc}$-limit points. This work provides a foundational framework for further exploration of convergence in Riesz spaces.}\\
\vspace{2mm}

{\bf{Keywords:} \rm Riesz space, rough convergence, $\mathbb{rc}$-convergence}
\vspace{2mm}

{\bf AMS Mathematics Subject Classification:} {\normalsize 46A40, 46B42, 40A05}

\section{Introduction}\label{Sec:1}
Convergence is one of the fundamental tool in mathematical analysis, and it plays a crucial role in the study of various properties related to sequences and nets, and their relation to other topological properties. It can be classified in two ways, a topological convergence and non-topological convergence. It can be classified in two ways, a topological convergence and non-topological convergence. These types of convergences are equally important in mathematical analysis. It is well-known that order convergence in an infinite dimensional Riesz spaces is non-topological \cite{Gor}, it however plays an important role in defining operator like order continuous operators in vector lattices, without necessitating any elaborate topological structure. Gorokhova show that order convergence on an infinite dimensional Riesz space is not topological \cite{Gor}. In finite dimensional Riesz spaces order convergence is equivalent to norm convergence and hence topological \cite{Andrew}. On the other hand, Ordman established that almost everywhere convergence cannot be characterized by a topology \cite{Ord66}. To explore different notions of convergence, researchers have considered the structure of convergence. Investigations were carried out to examine non-topological convergences by Hyland \cite{Hyland}. Also, Beattie and Butzmann presented useful results regarding the theory of convergence structures \cite{BB}. Ayd\i n et al. \cite{AEG} focused on full lattice convergence on Riesz spaces and introduced multiplicative order convergence on Riesz algebras \cite{AAydn2,AAydn3}, and O’Brien et al. \cite{Tr} presented the theory of net convergence structures, which is equivalent to filter convergence theory. 

Phu, expanding on progress, introduced a new concept of convergence in normed spaces, named rough convergence \cite{Phu2001}. This concept specifically was applied to sequences in finite-dimensional normed linear spaces. Thereafter, he applied this theory of rough convergence to introduce and study the rough continuity of linear operators \cite{Phu2002}, and provided a characterization of rough convergence in infinite-dimensional normed linear spaces \cite{Phu2003}. The application of rough convergence extends to weighted statistical convergence on locally solid Riesz spaces \cite{GB}, and it has been compared with different types of convergence in various spaces \cite{Aytar,AD}. The purpose of this paper is to introduce a comprehensive theory of rough convergence on Riesz spaces and give basic result about rough $\mathbb{c}$-limit points. While various forms of convergence exist, such as order and relatively uniform convergence on Riesz spaces, in order to maintain a broad perspective, we adopt the concept of $c$-convergence as defined in \cite{AEG}. Therefore, this study serves as a foundation for future specific definitions of rough convergence on Riesz spaces.

\begin{definition}\cite{Phu2001}
	For a given non-negative real number $r$, a sequence $(x_n)$ within a normed space $(X, \rVert\cdot\lVert)$ is described as {\em rough convergent} (or alternatively, {\em $r$-convergent}) to an element $x \in X$ if, for any arbitrary $\varepsilon > 0$, an index $n_\varepsilon \in \mathbb{N}$ can be found such that the inequality $\rVert x_n - x\lVert < r + \varepsilon$ is satisfied for all $n > n_\varepsilon$.
\end{definition}

This is equivalent to the condition $\limsup \rVert x_n - x\lVert\leq r$, where the constant $r$ signifies the degree of convergence for the sequence $(x_n)$. Clearly, any sequence that converges in the standard norm sense is also roughly convergent to the same limit for any $r\geq 0$. The converse, however, is not generally true. Thus, rough convergence may be regarded as a more relaxed criterion for convergence when compared to norm convergence.

A real vector space $E$ furnished with a partial order relation "$\leq$" (which is reflexive, antisymmetric, and transitive) is known as an {\em ordered vector space}, provided that the relation is compatible with the vector space operations. This compatibility requires that for any $x, y \in E$ with $x\leq y$, the relations $x+z\leq y+z$ and $\alpha x\leq\alpha y$ hold for every $z\in E$ and any non-negative scalar $\alpha \in \mathbb{R}_+$. A Riesz space is termed {\em Dedekind complete} if every non-empty subset that is bounded above possesses a supremum (which is equivalent to every non-empty subset bounded below having an infimum); see \cite{AB,ABPO,AAydn1,EG,Hu,LZ,Za} for further details.

The paper is organized as follows. Section $2$ is dedicated to a review of essential concepts related to the convergence of nets in the context of Riesz spaces and normed spaces. This section also introduces the notion of rough $\mathbb{c}$-convergence within Riesz spaces, formulated with respect to a general $\mathbb{c}$-convergence, thereby drawing an analogy to the theory of rough norm convergence. We present Example \ref{piece wise polynomial} to show that a roughly $\mathbb{c}$-convergent net is not necessarily fullatticification convergent. Additionally, Example \ref{c and rc coincide} is provided to illustrate scenarios where $\mathbb{c}$-convergence and rough $\mathbb{c}$-convergence are identical.

In Section $3$, we concentrate on building the foundational theory for rough $\mathbb{c}$-convergence. The circumstances under which rough $\mathbb{c}$-convergence behaves as a sequential convergence are identified in Proposition \ref{sequential topological convergence}, and Theorem \ref{rc is a convergence} validates that rough $\mathbb{c}$-convergence is a proper form of convergence in Riesz spaces. Proposition \ref{additive and unique} details the conditions that guarantee the uniqueness of the rough $\mathbb{c}$-limit. To demonstrate that rough $\mathbb{c}$-convergence lacks linearity, we use Example \ref{rc is not linear}. Proposition \ref*{LO are $p$-continuous} shows that lattice operations exhibit rough $\mathbb{c}$-continuity, and Theorem \ref{double nets} reveals several lattice-related properties of rough $\mathbb{c}$-convergence.

In Section $4$, we explore an extensive examination of the set of limit points associated with a rough $\mathbb{c}$-convergent net, denoted as $\mathcal{L}^\mathbb{r}_{x_\alpha}$. Proposition \ref{the inclusion about r} demonstrates that the inclusion $\mathcal{L}^\mathbb{r_1}_{x_\alpha}\subseteq \mathcal{L}^\mathbb{r_2}_{x_\alpha}$ holds true whenever $\mathbb{r}_{1}\leq\mathbb{r}_{2}$. Theorem \ref{order 2r} illustrates that $\sup\{|x-y|:x,y\in \mathcal{L}^\mathbb{r}_{x_\alpha}\}\leq2\mathbb{r}$, while Theorem \ref{suset inclusion} establishes a relationship between the limit points of a rough $\mathbb{c}$-convergent net and its subnet. Theorem \ref*{two properties} demonstrates the connection between the limit points of a rough $\mathbb{c}$-convergence and order boundedness. Proposition \ref{closedness} indicates that $(x_\alpha)_{\alpha\in A}$ is $\mathbb{c}$-closed. Finally, Theorem \ref{order convergence} proves the existence of a relationship between order convergence and rough $\mathbb{c}$-convergence.

\section{Preliminaries}\label{Sec:2}
First of all, the idea of rough convergence in relation to nets in Riesz spaces will be introduced. The definition of a net is not universally clear in the literature, as different definitions are considered. We provide a general definition of a net. A binary relation "$\leq$" on a set $A$ is referred to as a {\em preorder} if it satisfies reflexivity and transitivity. A non-empty set $A$, equipped with a preorder binary relation "$\leq$", is said to be a {\em directed upwards} (or simply a {\em directed set}) if for every pair of elements $x, y \in A$, there exists an element $z \in A$ such that $x \leq z$ and $y \leq z$. Let us recall the following notions from Definition 3.3.14 \cite{Rud}:
\begin{definition} \
	\begin{enumerate}[(1)]
		\item A function whose domain is a directed set is called a {\em net}, and denoted by $(x_\alpha)_{\alpha \in A}$, with its directed domain set $A$.
		\item A net $(y_\beta)_{\beta\in B}$ is considered a {\em subnet} of the net $(x_\alpha)_{\alpha \in A}$ in a nonempty set $X$ if there exists a function $\phi: B \rightarrow A$ such that $y_\beta=x_{\phi(\beta)}$ for all $\beta \in B$, and for each $\alpha \in A$ there exists $\beta_\alpha \in B$ such that $\alpha\leq \phi(\beta)$ for all $\beta\geq \beta_\alpha$. It is abbreviated as $(x_{\alpha_\beta})_{\beta\in B}$ in general.
	\end{enumerate}
\end{definition}

It should be noted that a sequence can be considered as a net indexed by natural numbers. Convergence of nets is a fundamental concept that appears in various fields and it has significant importance due to its ability to describe the behavior and properties of nets, sequences, series, functions, and other mathematical structures. However, there are some different approaches to the convergence of nets. In order to understand the convergence of nets, we consider the following structure. We take the following definition from  \cite[Def.1.1]{AEG}.
\begin{definition}\label{convergence}
	Let $X$ denote a set and $\mathbb{c}$ is a class of pairs $(C,c)$, where $C$ represents a net in $X$ and $c\in X$. Then, $\mathbb{c}$ is said to be a {\em convergence on $X$}, provided that:
	\begin{enumerate}
		\item[$(1)$] For any constant net $C\equiv c$ in $X$, $(C,c)\in\mathbb{c}$.
		\item[$(2)$] If $(C,c)\in\mathbb{c}$ and $A$ is a subnet of $C$, then $(A,c)\in\mathbb{c}$.
		\item[$(3)$] If a tail of a net $(C_\alpha)_{\alpha\in A}$ converges to $c$ (i.e., $(C_\alpha)_{\alpha\in\{\xi\in A: \xi\ge\alpha_0\}}$ converges to $c$ for an index $\alpha_0\in A$), then $(C_\alpha)_{\alpha\in A}$ converges to $c$.
	\end{enumerate}
\end{definition}

From now on, we denote the set $\mathbb{c}$ as $\{((x_\alpha)_{\alpha\in A},x):(x_\alpha)_{\alpha\in A} \ \text{\textit{is a net in}} \ X \ \text{\textit{and}} \ x\in X\}$. Also, note that we use the notation $x_\alpha\convc x$ as a shorthand for the statement $((x_\alpha)_{\alpha\in A},x)\in\mathbb{c}$. We provide a brief overview of basic concepts regarding the convergence of nets in Riesz spaces. For a more detailed exposition, interested readers are referred to the papers \cite{AEG,Tr}. In the context of this study, $\mathbb{c}$ and $\mathbb{t}$ denote convergences having properties fiven in Definition \ref{convergence} on sets $X$ and $Y$, respectively.
\begin{enumerate}
	\item[-] If the convergence $\mathbb{c}$ on the set $X$ is induced by a topology on
	$X$, then it is called {\em topological}.
	\item[-] A subset $K$ of $X$ is called {\em $\mathbb{c}$-closed} if every net $(x_\alpha)_{\alpha\in A}$ in $K$ satisfying $x_\alpha\convc x$ implies $x\in K$.
	\item[-] If $X$ is a real vector space, then $\mathbb{c}$ is called {\em linear} whenever the addition from $X\times X$ to $X$ and the scalar multiplication from $\mathbb{R}\times X$ to $X$ are continuous.
	\item[-] $\mathbb{c}$ is called {\em additive} whenever $x_\alpha\convc x$ and $y_\alpha\convc x$ implies that the net $(x_\alpha\pm y_\beta)_{(\alpha,\beta)\in A\times B}$ $\mathbb{c}$-convergent to $x\pm y$.
	\item[-] $\mathbb{c}$ is called {\em full convergence} if $X$ is an ordered vector space, $(x_\alpha)_{\alpha\in A}\convc \theta$ and $\theta\leq y_\alpha\leq x_\alpha$ hold for all $\alpha\in A$ implies $y_\alpha\convc \theta$.
	\item[-] $\mathbb{c}$ is called {\em lattice convergence} if $X$ is a Riesz space and $x_\alpha\convc x$ implies $|x_\alpha|\convc |x|$.
	\item[-] A net $(x_{\alpha})_{\alpha\in A}$ in $S$ is said to be {\em $\mathbb{sc}$-convergent to $x\in X$} whenever for any subnet $(x_{\alpha_\beta})_{\beta\in B}$ of $(x_{\alpha})_{\alpha\in A}$ there exists a sequence $\beta_n$ in $B$ such that $(x_{\alpha_{\beta_n}})_{n\in\mathbb{N}}\convc x$. Moreover, if $\mathbb{sc}=\mathbb{c}$ holds, then $\mathbb{c}$ is called {\em sequential}.
	\item[-] Let $\mathbb{c}$ be a convergence on a Riesz space $E$. The {\em fullification} $f\mathbb{c}$ of $\mathbb{c}$ is defined by $((y_\alpha)_{\alpha\in A}, y)\in f\mathbb{c}$ if there exists a net $((x_\alpha)_{\alpha\in A}, \theta)\in\mathbb{c}$ such that for all $\alpha\in A$ we have $|y_\alpha-y|\le x_\alpha$.
\end{enumerate}

Note that order convergence is very important notion on Riesz space because many topological properties can be defined in Riesz spaces without topology thanks to order convergence, and it provides the cornerstone for many studies. But, there exist various definitions of order convergence, an so we provide one of the most general definition of order convergence below. For a comprehensive discussion on various types of order convergence on Riesz spaces, we recommend that readers pay attention especially to \cite{AS,AB}. Order convergence, often abbreviated as $o$-convergence, stands as one of the most crucial convergence modes within a Riesz space.

\begin{definition}\label{monoton and ordr conv}
	Let $E$ be a Riesz space and $(x_\alpha)_{\alpha\in A}$ be a net in $E$.
	\begin{enumerate}
		\item If $(x_\alpha)_{\alpha\in A}$ satisfying $x_\alpha \leq x_\beta$ whenever $\beta\leq \alpha$ in a partial ordered set is called {\em decreasing}. In this case, it is denoted by $x_\alpha\downarrow$. Moreover, $x_\alpha\downarrow x$ means that $x_\alpha\downarrow$ and $\inf x_\alpha=x$ and we say that $(x_\alpha)$ converges monotonically decreasing to $x$.
		\item $(x_\alpha)_{\alpha\in A}$ called as {\em order convergent} to $x\in E$ (abbreviated as $x_\alpha\co x$) if there exists a net $y_\alpha\downarrow \theta$ with the same index set in $E$ such that $|x_\alpha-x|\le y_\alpha$ holds for all $\alpha\in A$.
	\end{enumerate}
\end{definition}

Thorough the paper, unless otherwise stated, the symbols $\mathbb{r}$ and $\mathbb{s}$ are considered as positive vectors (an element $x$ in a Riesz space is said to be {\em positive} if $x\geq\theta$, i.e., it is greater than or equal to zero) in Riesz spaces.
\begin{definition}
	Let $\mathbb{c}$ be a convergence in a Riesz space $E$. A net $(x_\alpha)_{\alpha\in A}$ is said to be {\em rough $\mathbb{c}$-convergent} (or {\em $\mathbb{rc}$-convergent}) to $x\in E$ if there exists a net $y_\alpha\cc \theta$ in $E$ such that
	$$
	|x_\alpha-x|\leq y_\alpha+\mathbb{r}
	$$
	holds for all $\alpha\in A$. This is indicated as $x_\alpha\rc x$.
\end{definition}

Let us note that our primary focus lies on the case of $\mathbb{r}\neq\theta$ because the case of $\mathbb{r}=\theta$ coincide with the classical convergence considered in \cite{AEG}. This particular case holds significant motivation. Consider nets $(z_\alpha)_{\alpha\in A}$ such that $z_\alpha\cc x$ and $|z_\alpha-x_\alpha|\leq \mathbb{r}$ holds for all $\alpha$. The following inequality 
$$
|x_\alpha - x| \leq |x_\alpha- z_\alpha|+|z_\alpha-x|\leq \mathbb{r}+|z_\alpha-x|,
$$
implies that $x_\alpha\rc x$ whenever $\mathbb{c}$ is a lattice convergence. This is not true for classical convergence. The concept of $\mathbb{rc}$-convergence is characterized by the introduction of a roughness degree denoted by $\mathbb{r}$. In Dedekind complete Riesz spaces, the $\mathbb{rc}$-convergence of a net $(x_\alpha)_{\alpha\in A}$ can be equivalently expressed as 
$$
\limsup|x_\alpha-x|\leq \mathbb{r}.
$$

Note that a net $(x_{\alpha})_{\alpha\in A}$ in $E$ is {\em $\mathbb{src}$-convergent to $x\in E$} whenever for any subnet $(x_{\alpha_\beta})_{\beta\in B}$ of $(x_{\alpha})_{\alpha\in A}$ there exists a sequence $\beta_n$ in $B$ such that $(x_{\alpha_{\beta_n}})_{n\in\mathbb{N}}\rc x$. Moreover, if $\mathbb{src}=\mathbb{rc}$ holds, then $\mathbb{rc}$-convergence is sequential.

It is clear that every fullification convergent net is $\mathbb{rc}$-convergent in Riesz spaces for all positive elements $\mathbb{r}$. On the other hand, fullification and $\mathbb{rc}$-convergence coincide in the case $\mathbb{r}=\theta$. However, the next example shows that this equivalence need not be true in general.
\begin{example}\cite[Ex.5.1(i)]{Hu} \label{piece wise polynomial} 
	Let a Riesz space $E$ consisting of all real valued continuous functions defined on $[0,1]$ which are piecewise polynomial with finitely many pieces. Take the following norm on $E$ (see for details \cite[Exam.2]{AEG}):
	$$
	\|f\|:=\|f\|_\infty+\sup\bigg\{\bigg|\frac{df}{dt}(t)\bigg|: t\in [0,1] \ \text{and} \ \frac{df}{dt}(t)\ \text{exists} \ \bigg\}
	$$ 
	for all $f\in E$. 
	
	Let $\mathbb{c}$ be the norm $\lVert\cdot\rVert$-convergence on $E$. Consider the sequence $(f_n)$ of functions, where $f_n:=|f-\frac{1}{n}|$ for all $n\in\mathbb{N}$ and a function
	\begin{equation*}
		f(t):= 
		\left\{
		\begin{array}{ll}
			0, & \text{if }\  0\leq t< 1/2 \\
			2t-1, & \text{if }\  1/2\leq t\leq 1
		\end{array}
		\right.
	\end{equation*}
	Thus, it can be easily seen that $\lVert f_n-f\rVert=\lVert |f-\frac{1}{n}|-|f|\rVert=4+\frac{1}{n}\nrightarrow 0$ holds. Therefore, $(f_n)$ is not fullatticification convergent. However, for the constant function $\mathbb{r}=4$, we have $f_n\rc f$.
\end{example}

One can notice that any net which converges in the linear $\mathbb{c}$-sense is also $\mathbb{rc}$-convergent to the same point in Riesz spaces for all positive elements $\mathbb{r}$. The next example illustrates that $\mathbb{rc}$-convergence and $\mathbb{c}$-convergence can be equivalent for a special case, even if $\mathbb{r} > 0$.
\begin{example}\label{c and rc coincide}
	Consider the convergence $\mathbb{c}$ defined as the weak convergence on the Banach lattice $E=L_\infty[0,1]$. Let $(r_n)$ be the sequence of Rademacher's functions, which are real-valued functions defined on $[0,1]$ such that $r_n(t)=\text{sgn}(\sin(2^n\pi t))$ for each $n\in\mathbb{N}$ and $t\in[0,1]$. It is evident that $r_n\cc 0$. Furthermore, we have $r_n\rc 0$ for the constant function $\mathbb{r}=1$ in $E$ due to $|r_n|\cc 1$.
\end{example}

\begin{remark}\label{full conv}\
	\begin{enumerate}[(i)]
		\item $\mathbb{rc}$-convergence is a full convergence for each $\mathbb{c}$-convergence on Riesz spaces. Indeed, assume that $(x_\alpha)_{\alpha\in A}\rc \theta$ and $\theta\le y_\alpha\le x_\alpha$ for all $\alpha\in A$. Then, there is a net $(z_\alpha)_{\alpha\in A}\cc \theta$ such that $|y_\alpha-\theta|=y_\alpha\le x_\alpha=|x_\alpha-\theta|\le z_\alpha +\mathbb{r}$ for all $\alpha\in A$. Hence, $(y_\alpha)_{\alpha\in A}\rc \theta$, and so $\mathbb{rc}$-convergence has the fullness property.
		
		\item $\mathbb{rc}$-convergence is a lattice convergence for each $\mathbb{c}$-convergence on Riesz spaces. Indeed, suppose that $(x_\alpha)_{\alpha\in A}\rc x$. Then, we have a net $(z_\alpha)_{\alpha\in A}\cc \theta$ such that $|x_\alpha-x|\le z_\alpha+\mathbb{r}$ for all $\alpha\in A$. It follows from the inequality $||x_\alpha|-|x||\le|x_\alpha-x|$ that we obtain $|x_\alpha|\rc|x|$. Therefore, $\mathbb{rc}$-convergence is a lattice convergence.
	\end{enumerate}
\end{remark}
\section{The theory of $\mathbb{rc}$-convergence}
In this section, we show that rough convergence provides some basic results in Riesz spaces.  Most of results are direct analogies of well-known facts of the theory of Riesz spaces. We include them for convenience of the reader. We denote the closure set 
$$
\text{cl}_{\mathbb{rc}}(F):=\{x\in E:\exists (x_\alpha)_{\alpha\in A} \ \text{in}\ F,\ x_\alpha\rc x\}
$$
for a Riesz subspace $F$ of a Riesz space $E$.
\begin{proposition}\label{sequential topological convergence}
	Let $E$ be a topological Riesz space. Then, the following statements are equivalent:
	\begin{enumerate}
		\item[(i)] The $\mathbb{rc}$-convergence is sequential$;$
		\item[(ii)] For any subspace $F$ and $x\in\text{cl}_{\mathbb{rc}}(F)$, there exists a sequence $(x_n)$ in $F$ such that $x_n\rc x$.
	\end{enumerate}
\end{proposition}

\begin{proof}
	($i)\implies(ii)$: Let $x$ be an element in the $\mathbb{rc}$-closure of $F$, denoted $\text{cl}_{\mathbb{rc}}(F)$. By definition, this implies there is a net $(x_\alpha)$ composed of elements from $F$ that $\mathbb{rc}$-converges to $x$. Since $\mathbb{rc}$-convergence is sequential, we can extract a sequence $(x_{\alpha_n})$ from the net $(x_\alpha)$ that also $\mathbb{rc}$-converges to $x$. This provides the required sequence.
	
	($ii)\implies(i)$: First, we will establish that $\mathbb{src}$-convergence and $\mathbb{rc}$-convergence are equivalent. Let a net $(x_\alpha)$ $\mathbb{rc}$-converge to $x$, and consider any subnet $(x_{\alpha_\beta})$. From the definition of $\mathbb{rc}$-convergence, it follows directly that this subnet also $\mathbb{rc}$-converges to $x$. Consequently, $x$ belongs to the $\mathbb{rc}$-closure of the set of the subnet's elements, $\{x_{\alpha_\beta}:\beta\in B\}$. By applying assumption $(ii)$, we can find a sequence $(x_{\alpha_{\beta_n}})$ within this subnet that $\mathbb{rc}$-converges to $x$. This confirms that the original net $(x_\alpha)$ is $\mathbb{src}$-convergent to $x$.
	
	For the reverse direction, assume the net $(x_\alpha)$ is $\mathbb{src}$-convergent to $x$. We will argue by contradiction. Suppose that the net does not $\mathbb{rc}$-converge to $x$, i.e., $(x_{\alpha})_{\alpha\in A} \stackrel{\mathbb{rc}}{\nrightarrow} x$. This implies the existence of a neighborhood $U$ of $x$ and a subnet $(x_{\alpha_\beta})$ such that no element of the subnet is in $U$. However, because the original net $(x_\alpha)$ is $\mathbb{src}$-convergent to $x$, this subnet must contain a sequence $(x_{\alpha_{\beta_n}})$ which $\mathbb{rc}$-converges to $x$. This convergence means that for some index $n_0$, all terms $x_{\alpha_{\beta_n}}$ with $n \ge n_0$ must lie within $U$. This creates a contradiction, as we established that no element of the subnet $(x_{\alpha_\beta})$ could be in $U$. Therefore, our initial supposition is false, and it must be that $x_\alpha \rc x$.
\end{proof}

The following example shows that not every sequence has to be $\mathbb{rc}$-convergent and $\mathbb{rc}$-convergence does not necessarily imply $\mathbb{c}$-convergence in general.
\begin{example}
	Consider the Euclidean space $E:=\mathbb{R}^2$ with the lexicographic order (i.e., $(x_1,y_1)\leq(x_2,y_2)$ if and only if $x_1<x_2$ or $x_1=x_2$ and $y_1\leq y_2$) equipped with the $\mathbb{c}$-convergence denoted as monotonically decreasing convergence defined in Definition \ref{monoton and ordr conv}$(1)$. 
	\begin{enumerate}
		\item[(a)] Consider the net $(x_\alpha)_{\alpha\in \mathbb{N}}$ in $E$ defined by $x_\alpha:=(-\alpha,\alpha)$ for every $\alpha\in\mathbb{N}$. Then, it can be easily seen that $(x_\alpha)_{\alpha\in A}$ is not $\mathbb{rc}$-convergent.
		
		\item[(b)] Take the net $(x_\alpha)_{\alpha\in \mathbb{N}}$ defined by $x_\alpha:=\left(\frac{1}{\alpha}, \frac{1}{\alpha}\right)$ for each $\alpha\in\mathbb{N}$. We show that $(x_\alpha)_{\alpha\in \mathbb{N}}$ is $\mathbb{rc}$-convergent to $x:=(0,0)$ but not $\mathbb{c}$-convergent to $x$. Indeed, to see the $\mathbb{rc}$-convergence of $(x_\alpha)_{\alpha\in A}$, consider a net $(y_\alpha)_{\alpha\in A}$ defined by $y_\alpha:=\left(0,\frac{2}{\alpha}\right)$ for all $\alpha\in\mathbb{N}$. It is straightforward to see that $(y_\alpha)_{\alpha\in A}\cc x$. On the other hand, for each $\alpha\in \mathbb{N}$, we have
		$$
		|x_\alpha-x|=\left(\frac{1}{\alpha}, \frac{1}{\alpha}\right)\leq \left(0+1,\frac{2}{\alpha}+0\right)= y_\alpha+(1,0)
		$$
		for all $\alpha\in\mathbb{N}$ and the positive element $\mathbb{r}=(1,0)$ in $E$. Hence, we obtain the desired condition of $\mathbb{rc}$-convergent of the net $(x_\alpha)_{\alpha\in A}$, i.e., $x_\alpha\rc x$. However, it is clear that the net $\left(\frac{1}{\alpha}, \frac{1}{\alpha}\right)$ is decreasing but not convergent to $x=(0,0)$. Therefore, $(x_\alpha)_{\alpha\in A}$ is not $\mathbb{c}$-convergent to $x=(0,0)$.
	\end{enumerate}
\end{example}

The next theorem show that $\mathbb{rc}$-convergence satisfies the conditions of convergence.
\begin{theorem}\label{rc is a convergence}
	Let $E$ be Riesz space equipped with a convergence $\mathbb{c}$. Then, $\mathbb{rc}$-convergence is also a convergence in Riesz spaces.
\end{theorem}

\begin{proof}
	Let $\mathbb{c}$ be a convergence on a Riesz space $E$. We prove that $\mathbb{rc}$-convergence possesses the three features of Definition \ref{convergence} on $E$.
	
	$(1)$ Take a constant net $x_\alpha:=x$. Then, we have $|x_\alpha-x|=|x-x|=\theta$. Consider an arbitrary non-negative net $y_\alpha\cc \theta$ in $E$. Since $\theta\leq y_\alpha+\mathbb{r}$ satisfies for all $\alpha\in A$ and for any positive vector $\mathbb{r}\in E_+$, it follows that $|x_\alpha-x|\leq y_\alpha+\mathbb{r}$ holds for every $\alpha\in A$. Therefore, we obtain $x_\alpha\rc x$.
	
	$(2)$ Suppose that $(x_{\alpha_\beta})_{\beta\in B}$ is any subnet of $x_\alpha\rc x$. Since $x_\alpha\rc x$, we have a net $y_\alpha\cc \theta$ in $E$ such that $|x_\alpha-x|\leq y_\alpha+\mathbb{r}$ for each $\alpha\in A$, and so $|x_{\alpha_\beta}-x|\leq y_{\alpha_\beta}+\mathbb{r}$ holds for all $\beta\in B$. It follows from $y_\alpha\cc \theta$ that $y_{\alpha_\beta}\cc \theta$ as $\beta$ varies over $B$ because $\mathbb{c}$ is a convergence on $E$. Thus, we obtain $x_{\alpha_\beta}\rc x$.
	
	$(3)$ Assume that $(x_\alpha)_{\alpha \in A,\alpha\geq \alpha_0} \rc x$ for some $\alpha_0\in A$. It follows that there exists a net $(y_\alpha)_{\alpha\in A}\cc \theta$ in $E$ such that $|x_\alpha - x|\leq y_\alpha + \mathbb{r}$ for all $\alpha \geq \alpha_0$. We need to show that $(x_\alpha)_{\alpha\in A}\rc x$. To do this, we find a net $(z_\alpha)_{\alpha\in A}\cc \theta$ such that $|x_\alpha - x|\leq z_\alpha + \mathbb{r}$ for all $\alpha\in A$. Let's define the net $(z_\alpha)_{\alpha\in A}$ as follows:
	\[z_\alpha:=\begin{cases} 
		y_\alpha, & \alpha\geq \alpha_0 \\
		|x_\alpha-x|,& otherwise 
	\end{cases}
	\]
	It follows that $(z_\alpha)_{\alpha \in A,\alpha\geq \alpha_0}=(y_\alpha)_{\alpha \in A,\alpha\geq \alpha_0}\cc \theta$. Thus, by the third property of Definition \ref{convergence}, we have $(z_\alpha)_{\alpha\in A}\cc \theta$ and $|x_\alpha - x|\leq z_\alpha + \mathbb{r}$ holds for all $\alpha\in A$, and so we obtain $(x_\alpha)_{\alpha\in A}\rc x$.
\end{proof}

One of the fundamental differences between rough convergence and other convergences on Riesz spaces such as order convergence is the that limit of a roughly convergent net is not necessarily unique. However, we provide conditions under which a net have a unique rough limit in the following two results.
\begin{proposition}\label{additive and unique}
	Suppose that $\mathbb{rc}$-convergence satisfies the additive property on a Riesz space $E$. Then, the $\mathbb{rc}$ limit of a net is unique if and only if every constant net is uniquely $\mathbb{rc}$-convergent in $E$.
\end{proposition}

\begin{proof}
	It is enough to show the sufficient case because the necessary case is trivial. Suppose that $x_\alpha\rc x_1$ and $x_\alpha\rc x_2$ in $E$. Then, it follows from the additivity of $\mathbb{rc}$-convergence that we have
	$$
	\theta=(x_\alpha-x_\alpha)\rc (x_1-x_2)
	$$
	Therefore, by the assumption, we get the desired result, $x_1=x_2$.
\end{proof}

The following proposition is a $\mathbb{rc}$-version of \cite[Prop.2.2]{AEG} with a similar proof.
\begin{proposition}\label{inequalities}
	Let $\mathbb{rc}$-convergence be additive on a Riesz space $E$. Then, the following statements are equivalent$:$
	\begin{enumerate}
		\item[$(i)$] the limit of $\mathbb{rc}$-convergent net is unique;
		\item[$(ii)$] $x_\alpha\rc x$,  $z_\alpha\rc z$, and $x_\alpha\ge z_\alpha$ for all $\alpha\in A$ implies $x\ge z$;
		\item[$(iii)$] The positive cone $E_+$ is $\mathbb{rc}$-closed in $E$.
	\end{enumerate}
\end{proposition}
\begin{proof}
	$(i)\Longrightarrow(ii)$ \  Assume that $x_\alpha\rc x$, $z_\alpha\rc z$ and $z_\alpha\leq x_\alpha$ hold for all $\alpha\in A$. By using the additivity of $\mathbb{rc}$-convergence on $E$, we have $(z_\alpha-x_\alpha)_{\alpha\in A}\rc z-x$. Moreover, since $\mathbb{rc}$-convergence has the lattice property, we obtain $|z_\alpha-x_\alpha|\rc |z-x|$. It follows from the equalities $\theta=(z_\alpha-x_\alpha)^+=\frac{1}{2}(z_\alpha-x_\alpha+|z_\alpha-x_\alpha|)$ and $(z-x)^+=\frac{1}{2}(z-x+|z-x|)$ that $(\theta_\alpha)_{\alpha\in A}\rc (z-x)^+$. By considering Proposition \ref{additive and unique}, we obtaion that $(z-x)^+=\theta$, and so we get $z\leq x$.
	
	$(ii)\Longrightarrow(i)$ \ Suppose that $\mathbb{rc}$-convergence is not unique. Thus, by Proposition \ref{additive and unique}, there exists a constant net $x_\alpha:=x$ such that $x_\alpha\rc z$ for some $z\neq x$. Then, we have $\theta=x-x\rc x-z$. Now, take two new constant nets $u_\alpha=v_\alpha\equiv\theta$. Thus, we have $u_\alpha\rc \theta$ and $v_\alpha\rc |x-z|\neq\theta$. Therefore, we get the contradiction $z\nleq x$ with $(ii)$. Hence, $\mathbb{rc}$-convergence is unique.
	
	$(ii)\Longrightarrow(iii)$\ Let $(x_\alpha)_{\alpha\in A}$ be a net in $E_+$ such that $x_\alpha\rc x$. Consider the constatnt net $y_\alpha\equiv\theta\rc \theta$. Then, it follows that $\theta=y_\alpha\leq x_\alpha\rc x$, and so we get $\theta\le x$.
	
	$(iii)\Longrightarrow(ii)$\ Assume that $x_\alpha\rc x$, $z_\alpha\rc z$, and $z_\alpha\leq x_\alpha$ for all $\alpha\in A$. By additivity of $\mathbb{rc}$-convergence, we have $\theta\leq x_\alpha-z_\alpha\rc x-z$. Since $E_+$ is $\mathbb{rc}$-closed, we obtain the desired result $x-z\in E_+$, i.e., $z\leq x$.
\end{proof}

\begin{proposition}
	Let $\mathbb{c}$ be a full lattice convergence on a Riesz space $E$. If $x_\alpha\rc x$ and $y_\alpha\rc x$ in $E$, then any nets $(z_\alpha)_{\alpha\in A}$ that belong to a collection containing $x_\alpha$ and $y_\alpha$ is also $\mathbb{rc}$-convergent to $x$.
\end{proposition}

\begin{proof}
	Assume that $x_\alpha\rc x$ and $y_\alpha\rc x$ in $E$ and $z_\alpha\in\{x_\alpha,y_\alpha\}$ for $\alpha\in A$. We want to show that $z_\alpha\rc x$. It follows from $x_\alpha\rc x$ and $y_\alpha\rc x$ that there exist two nets $u_\alpha\cc \theta$ and $v_\alpha\cc \theta$ in $E$ such that $|x_\alpha-x|\leq u_\alpha+\mathbb{r}$ and $|y_\alpha-x|\leq v_\alpha+\mathbb{r}$ for all $\alpha\in A$. Take the net $(w_\alpha)_{\alpha\in A}$ defined as $w_\alpha:=|z_\alpha-x|$. Therefore, we have two cases to consider:
	
	Case 1: If $z_\alpha = x_\alpha$, then $w_\alpha = |x_\alpha - x|$. Following from $|x_\alpha-x|\leq u_\alpha+\mathbb{r}$, we have $w_\alpha\leq u_\alpha+\mathbb{r}$ for all $\alpha\in A$.
	
	Case 2: If $z_\alpha = y_\alpha$, then $w_\alpha = |y_\alpha - x|$. By $|y_\alpha-x|\leq v_\alpha+\mathbb{r}$, we have $w_\alpha\leq v_\alpha+\mathbb{r}$ for each $\alpha\in A$.
	
	Therefore, we have $w_\alpha\leq u_\alpha\vee v_\alpha+\mathbb{r}$ for each $\alpha\in A$ and $z_\alpha\in\{x_\alpha,y_\alpha\}$. By appying \cite[Thm.1]{AEG}, we have $u_\alpha\vee v_\alpha\cc \theta$ because $\mathbb{c}$ is full lattice convergence. Therefore, we have shown that $w_\alpha\rc \theta$, and so we conclude that $z_\alpha\rc x$.
\end{proof}

It is well known that the linearity property is a very important concept for convergences. However, as can be seen in the next example, the addition of $\mathbb{rc}$-convergence does not hold in general.
\begin{example}\label{rc is not linear}
	Let $E$ be the Riesz space $\mathbb{R}$ with the usual order and operations, and positive elements $x,y$. Consider the index set $A = \mathbb{N}$ with the natural order. Define the nets  $x_\alpha:=x$ if $\alpha$ is even and $-x$ if $\alpha$ is odd, and similarly $y_\alpha:=y $ if $\alpha$ is even and $-y$ if $\alpha$ is odd. Then, we have $x_\alpha\rc x$ and $y_\alpha\rc y$ for the roughness degree $\mathbb{r}:=2\max\{x,y\}$.
	
	Now, we show that $(x_\alpha+y_\alpha)_{\alpha\in A}$ is not $\mathbb{rc}$-convergent to $x+y$ with roughness degree $\mathbb{r}$. Note that $(x_\alpha+y_\alpha)_{\alpha\in A}$ takes the values $x+y$ for even indices and $-(x+y)$ for odd indices. Suppose that $x_\alpha+y_\alpha\rc x+y$. By the definition of $\mathbb{rc}$-convergence, there must exist a net $z_\alpha\cc \theta$ in $E$ such that $|x_\alpha+y_\alpha-(x+y)|\leq z_\alpha+\mathbb{r}$ for all $\alpha\in A$. However, no matter how the net $(z_\alpha)_{\alpha\in A}$ is chosen, there will always be infinitely many indices $\alpha$ for which 
	$$
	|x_\alpha+y_\alpha-(x+y)|=2(x+y)>\mathbb{r}.
	$$
	This means that $(x_\alpha+y_\alpha)_{\alpha\in A}$ is not $\mathbb{rc}$-convergent to $x+y$ with roughness degree $\mathbb{r}$.
\end{example}

\begin{theorem}\label{full}
	Let $\mathbb{c}$ be a linear convergence on a Riesz space $E$. Then, the following statements are true$:$
	\begin{itemize}
		\item[(i)] If $(x_\alpha)_{\alpha\in A}\rc x$ and $(y_\beta)_{\beta\in B}\convsc y$, then $(x_\alpha+y_\beta)_{(\alpha,\beta)\in A\times B}\xrightarrow{\mathbb{(r+s)c}} x+y$;
		\item[(ii)] If $(x_\alpha)_{\alpha\in A}\rc x$ and $(t_\gamma)_{\gamma\in\Gamma}\to t$ in the standard topology on $\mathbb{R}$, then $(t_\gamma x_\alpha)_{(\gamma,\alpha)\in \Gamma\times A}\xrightarrow{(m\mathbb{r)c}} tx$ for some $m\in\mathbb{R}^+$;
		\item[(iii)] The $\mathbb{c}$-convergence implies $\mathbb{rc}$-convergence whenever $\mathbb{c}$ is a lattice convergence.
	\end{itemize}
\end{theorem}

\begin{proof}	
	$(i)$ Consider nets $(x_\alpha)_{\alpha\in A}\rc x$ and $(y_\beta)_{\beta\in B}\rc y$ in $E$. Then, there exist nets $(z_\alpha)_{\alpha\in A}\cc \theta$ and $(w_\beta)_{\beta\in B}\cc \theta$ in $E$, such that 
	$$
	|x_\alpha-x|\le z_\alpha +\mathbb{r}\ \text{and} \ |y_\beta-y|\le w_\beta+\mathbb{s} 
	$$
	for all $\alpha \in A$ and $\beta\in B$. It follows
	$$
	|(x_\alpha+y_\beta)-(x+y)|\le|x_\alpha-x|+|y_\beta-y|\le z_\alpha+w_\beta+\mathbb{r+s}
	$$
	for each $\alpha\in A$ and $\beta\in B$. Since the convergence $\mathbb{c}$ is linear then $(z_\alpha+w_\beta)_{(\alpha,\beta)\in A\times B}\cc \theta$, and so we have $(x_\alpha+y_\beta)_{(\alpha,\beta)\in A\times B}\xrightarrow{\mathbb{(r+s)c}} x+y$.
	
	$(ii)$ We have some nets $(u_\alpha)_{\alpha\in A}\cc \theta$ in $E$, such that $|x_\alpha-x|\le u_\alpha +\mathbb{r}$ for all $\alpha \in A$. Now, take $\alpha_0\in A$ and let $\gamma_0\in\Gamma$ and $\rho\in \mathbb{R}$ such that $|t_\gamma|\le|t|+\rho$ for all $\gamma\ge\gamma_0$. Thus, we have
	\begin{eqnarray*}
		|t_\gamma x_\alpha-tx|&\le&|t_\gamma||x_\alpha-x|+|t_\gamma x -tx|\\ &\le& (|t|+\rho)(|x_\alpha-x|)+|t_\gamma -t||x|\\&\le&(|t|+\rho)(u_\alpha+\mathbb{r})+|t_\gamma -t||x| \\&=&\big((|t|+\rho)u_\alpha+|t_\gamma -t||x|\big)+(|t|+\rho)\mathbb{r}
	\end{eqnarray*}
	for each $\alpha\geq\alpha_0$ and $\gamma\geq\gamma_0$. It follows from the linearity of $\mathbb{c}$-convergence that
	$$
	\big((|t|+\rho)u_\alpha+|t_\gamma-t||x|\big)_{(\gamma,\alpha)\ge(\gamma_0,\alpha_0)\in \Gamma\times A}\cc \theta.
	$$
	Therefore, we obtain that $(t_\gamma x_\alpha)_{(\gamma,\alpha)\in \Gamma\times A}\xrightarrow{(m\mathbb{r)c}} tx$ for the positive scalar $m:=|t|+\rho$ because $\mathbb{rc}$ is a convergence on $E$.
	
	$(iii)$ Assume that $(z_\alpha)_{\alpha\in A}\cc z$. Then, we have $(z_\alpha-z)_{\alpha\in A}\cc \theta$ because $\mathbb{c}$ is linear. Moreover, it follows from the lattice property of $\mathbb{c}$-convergence that $|z_\alpha-z|\cc \theta$. Thus, by the inequality $|z_\alpha-z|\leq|z_\alpha-z|+\mathbb{r}$, we obtain $(z_\alpha)_{\alpha\in A}\rc z$.
\end{proof}

\begin{question}
	Under which conditions does $\mathbb{rc}$-convergence imply $\mathbb{c}$-convergence?
\end{question}

\begin{remark}
	If $x_\alpha\rc x$ in a Riesz space $E$, then $x_\alpha\cc x+\mathbb{r}$ whenever $x+\mathbb{r}\leq x_\alpha$ satisfies for all $\alpha\in A$ and $\mathbb{c}$ is a full lattice convergence on $E$. Indeed, assume that $(x_\alpha)_{\alpha\in A}\rc x$. Then, we have a net $(z_\alpha)_{\alpha\in A}\cc \theta$ in $E$ such that $|x_\alpha-x|\le z_\alpha +\mathbb{r}$ for all $\alpha\in A$. Thus, we have
	$$
	\theta\leq x_\alpha-x-\mathbb{r}\le z_\alpha.
	$$
	It follows from the full property of $\mathbb{c}$-convergence    that we get the desired result.
\end{remark}

\begin{proposition}\label{first thm}
	Let $(x_\alpha)_{\alpha\in A}$ be a net in a Riesz space $E$. If there exists another net $(z_\alpha)_{\alpha\in A}$ in $E$ such that $z_\alpha\rc x$ and $|x_\alpha-z_\alpha|\leq \mathbb{t}$ for all $\alpha\in A$, then $x_\alpha\xrightarrow{\mathbb{(r+t)c}} x$.
\end{proposition}

\begin{proof}
	Suppose that there exists a net $(z_\alpha)_{\alpha\in A}$ in $E$ which satisfies $z_\alpha\rc x$ and $|x_\alpha-z_\alpha|\leq \mathbb{t}$ for all $\alpha\in A$. Then, there exists $(y_\alpha)_{\alpha\in A}\cc \theta$ such that $|z_\alpha-x|\leq y_\alpha+\mathbb{r}$ for all $\alpha\in A$. So we have
	$$
	|x_\alpha-x|\leq |x_\alpha-z_\alpha|+|z_\alpha-x|\leq \mathbb{t}+y_\alpha+\mathbb{r}
	$$
	for each $\alpha\in A$. Thus, we get $x_\alpha\xrightarrow{\mathbb{(r+t)c}} x$.
\end{proof}

\begin{question}
	Is the converse of Proposition \ref{first thm} true?
\end{question}

The lattice operations are $\mathbb{rc}$-continuous in the following sense.
\begin{proposition}\label{LO are $p$-continuous}
	Let $\mathbb{c}$ be a convergence, and $(x_\alpha)_{\alpha \in A}$ and $(y_\beta)_{\beta \in B}$ be two nets in a Riesz space $E$. If $x_\alpha\rc x$ and $y_\beta\xrightarrow{\mathbb{tc}} y$, then $(x_\alpha\vee y_\beta)_{(\alpha,\beta)\in A\times B}\xrightarrow{\mathbb{(r+t)c}} x\vee y$.
\end{proposition}

\begin{proof}
	From the assumption, there exist two nets $(z_\alpha)_{\alpha\in A}$ and $(w_\beta)_{\beta\in B}$ in $E$ satisfying $z_\alpha\cc \theta$ and $w_\beta\cc\theta$ in $E$ such that $|x_\alpha-x|\leq z_\alpha+\mathbb{r}$ and $|y_\beta-y|\leq w_\beta+\mathbb{t}$ for all $\alpha\in A$ and $\beta\in B$. It follows from the inequality $|a\vee b-a\vee c|\leq |b-c|$ that 	
	\begin{eqnarray*}
		|x_\alpha \vee y_\beta - x\vee y|&=&|x_\alpha \vee y_\beta -x_\alpha \vee y+x_\alpha \vee y- x\vee y|\\&\leq& |x_\alpha \vee y_\beta -x_\alpha \vee y|+|x_\alpha \vee y- x\vee y|\\&\leq&|y_\beta -y|+|x_\alpha-x|\\&\leq& z_\alpha+\mathbb{r}+w_\beta+\mathbb{t}
	\end{eqnarray*}
	is satisfied for all $\alpha\in A$ and $\beta\in B$. Since $(z_\alpha+w_\beta)\cc \theta$, then $x_\alpha\vee y_\beta\xrightarrow{\mathbb{(r+t)c}}x\vee y$ holds. 
\end{proof}

The statements of the following theorem is a rough $\mathbb{c}$-convergent version of \cite[Thm.2]{AEG}, and so the proofs of them will be omitted.
\begin{theorem}\label{double nets}
	Let $\mathbb{rc}$-convergence be linear on a Riesz space $E$. Then, the following statements hold$:$
	\begin{enumerate}[(i)]
		\item \ $x_\alpha\rc x$ iff $(x_\alpha-x)\rc \theta$ iff  $|x_\alpha-x|\rc \theta$;
		\item \ if $x_\alpha\rc x$, then $x_\alpha^+\rc x^+$$;$
		\item \ if $x_\alpha\rc x$, then $x_\alpha^-\rc x^-$$;$
		\item \ if $x_\alpha\rc x$ and $y\in X$, then $x_\alpha\vee y\rc x\vee y$$;$
		\item \ if $x_\alpha\rc x$ and $y\in X$, then $x_\alpha\wedge y\rc x\wedge y$$;$
		\item \ if $(x_\alpha)_{\alpha\in A}\rc x$ and $(y_\beta)_{\beta\in B}\rc y$, then $(x_\alpha\wedge y_\beta)_{(\alpha,\beta)\in A\times B}\rc  x\wedge y$$;$
		\item\ if $(x_\alpha)_{\alpha\in A}\rc x$ and $(y_\beta)_{\beta\in B}\rc y$, then $(x_\alpha\vee y_\beta)_{(\alpha,\beta)\in A\times B}\rc x\vee y$.
	\end{enumerate}
\end{theorem}

It is clearly observed that the statement $(i)$-$(v)$ of Theorem \ref{double nets} hold for a linear topology $\tau$ on a Riesz space $E$ and statement $(vi)$-$(vii)$ satisfy for locally full linear topology on Riesz spaces.
\section{Rough $\mathbb{c}$-limit points of nets}

It should be noted that the $\mathbb{r}$-limit point of a net $(x_\alpha)_{\alpha \in A}$ is not generally unique, especially for strictly positive elements of $\mathbb{r}$. Thus, we denote the set of all $\mathbb{r}$-limit points of a net $(x_\alpha)_{\alpha\in A}$ by
$$
\mathcal{L}^\mathbb{r}_{x_\alpha}:=\{x \in E: x_\alpha\rc x\}.
$$
It means that for each $x\in \mathcal{L}^\mathbb{r}_{x_\alpha}$ there exists a net  $y_\alpha\cc \theta$ such that $|x_\alpha-x| \leq y_\alpha+\mathbb{r}$ satisfies for all $\alpha\in A$. One can observe that $\mathcal{L}^\mathbb{r}_{x_\alpha}\cap \mathcal{L}^\mathbb{r}_{z_\alpha}=\emptyset$ can be hold for any $\mathbb{rc}$-convergent nets $(x_\alpha)_{\alpha\in A}$ and $(z_\alpha)_{\alpha\in A}$ in a Riesz space.	The next statement is a straightforward consequence of definitions, and hence its proof is not provided.
\begin{proposition}\label{the inclusion about r}
	Let $(x_\alpha)_{\alpha\in A}$ be a net in a Riesz space $E$. For elements  $\mathbb{r}_{1},\mathbb{r}_{2} \in E_+$ with $\mathbb{r}_{1}\leq\mathbb{r}_{2}$, the inclusion $\mathcal{L}^\mathbb{r_1}_{x_\alpha}\subseteq \mathcal{L}^\mathbb{r_2}_{x_\alpha}$ holds.
\end{proposition}

\begin{theorem}\label{order 2r}
	Let $(x_\alpha)_{\alpha\in A}$ be a net in a Riesz space $E$. Then, we have $\sup\{|x-y|:x,y\in \mathcal{L}^\mathbb{r}_{x_\alpha}\}\leq2\mathbb{r}$ whenever the supremum exists for an additive $\mathbb{c}$-convergence, and also there is no smaller bound.
\end{theorem}

\begin{proof}
	Suppose that $\sup\{|x-y|:x,y\in \mathcal{L}^\mathbb{r}_{x_\alpha}\}$ exists and equal to $w\in E_+$. By the contrary fact, assume that $w\nleq 2\mathbb{r}$. Hence, there exist some elements $x,y\in \mathcal{L}^\mathbb{r}_{x_\alpha}$ satisfying $|x-y|\nleq 2\mathbb{r}$. On the other hand, for arbitrary $x,y\in \mathcal{L}^\mathbb{r}_{x_\alpha}$, there exist some nets $u_\alpha\cc \theta$ and $v_\alpha\cc \theta$ such that
	$$
	|x_\alpha-x|\leq u_\alpha+\mathbb{r} \ \ \ \text{and} \ \ \ |x_\alpha-y|\leq v_\alpha+\mathbb{r}
	$$
	hold for all $\alpha\in A$. Thus, it follows that
	$$
	|x-y|\leq |x_\alpha-x|+|x_\alpha-y|\leq (u_\alpha+v_\alpha)+2\mathbb{r}.
	$$
	Since $\mathbb{c}$ is an additive convergence, we get $|x-y|\leq2\mathbb{r}$. Therefore, $2\mathbb{r}$ is an upper bound of the given set in the assumption.
	
	Now, take a net $(z_\alpha)_{\alpha\in A}$ in $E$ such that $z_\alpha\cc z$. Then, we have the following inequality
	$$
	|z_\alpha-w|\leq |z_\alpha-z|+|z-w|\leq |z_\alpha-z|+\mathbb{r}
	$$
	for every $w\in \mathcal{A}_z:=\{w\in E:|w-z|\leq\mathbb{r}\}$. Therefore, we get $\mathcal{A}_z= \mathcal{L}^\mathbb{r}_{z_\alpha}$. It follows from $\sup\{|x-y|:x,y\in \mathcal{L}^\mathbb{r}_{x_\alpha}\}=2\mathbb{r}$ that the upper bound $2\mathbb{r}$ of $\mathbb{r}$-limit set cannot be decreased anymore.
\end{proof}

The result in Theorem \ref{order 2r} indicates that when $\mathbb{r}=\theta$, for any net $(x_\alpha)_{\alpha\in A}$ in a Riesz space $E$, it's evident that $\sup\{|x-y|:x,y\in \mathcal{L}_{x_\alpha}^\mathbb{r}\}=\theta$. This implies that $\mathcal{L}_{x_\alpha}^\mathbb{r}$ is either an empty set or contains only one element. Additionally, the uniqueness of the rough $\mathbb{c}$-limit becomes apparent in this case. 

The following result is straightforward by considering the property $(2)$ of Definition \ref{convergence}, and so we omit its proof.
\begin{theorem}\label{suset inclusion}
	Let $(x_{\alpha_\beta})_{\beta\in B}$ be a subnet of a net $(x_\alpha)_{\alpha\in A}$ in a Riesz space $E$. Then, we have the following inclusion $\mathcal{L}^{\mathbb{r}}_{x_\alpha}\subseteq \mathcal{L}^{\mathbb{r}}_{x_{\alpha_\beta}}$.
\end{theorem}

It is well known that unbounded sequences in normed spaces lack rough limit points. This observation also holds true in the context of rough $\mathbb{c}$-convergence. This similarity extends to rough convergence can be observed. In the following three results, we present the rough $\mathbb{c}$-convergent version of the basic results related to rough convergence, specifically focusing on order boundedness, closeness and convexity.
\begin{theorem}\label{two properties}
	Let $(x_\alpha)_{\alpha\in A}$ be a net in a Riesz space $E$. Then, the following statements hold:
	\begin{itemize}
		\item[(i)] If $(x_\alpha)_{\alpha\in A}$ is order bounded and $E$ is  Dedekind complete, then there is a non-negative element $\mathbb{r}$ such that $\mathcal{L}^\mathbb{r}_{x_\alpha}\neq\emptyset$ for any convergence $\mathbb{c}$ on $E$;
		\item[(ii)] If $\mathcal{L}^\mathbb{r}_{x_n}\neq\emptyset$ for some non-negative elements $\mathbb{r}$ and $\mathbb{c}$ is the monotonically decreasing convergence on $E$, then the sequence $(x_n)$ is order bounded.
	\end{itemize}
\end{theorem}

\begin{proof}
	$(i)$ Since $(x_\alpha)_{\alpha\in A}$ is order bounded in a Dedekind complete Riesz space $E$, then the supremum of $(|x_\alpha|)_{\alpha\in A}$ exists and equal to a non-negative element $\mathbb{r}\in E$. So, for any nonnegative net $\theta\leq y_\alpha\cc \theta$ in $E$, we have $|x_\alpha-\theta|\leq \mathbb{r}+y_\alpha$ for all $\alpha\in A$. Hence, we obtain that $\theta\in \mathcal{L}^\mathbb{r}_{x_\alpha}$, i.e., we obtain $\mathcal{L}^\mathbb{r}_{x_\alpha}\neq\emptyset$.
	
	$(ii)$ Assume that $\mathcal{L}^\mathbb{r}_{x_n}\neq\emptyset$ satisfies for some non-negative element $\mathbb{r}\in E$. Thus, there exists some elements $x\in \mathcal{L}^\mathbb{r}_{x_n}$ such that, for some sequence $y_n\downarrow \theta$, we have $|x_n-x|\leq y_n+\mathbb{r}$ for each $n\in \mathbb{N}$. Then, we have $|x_n-x|\leq y_n+\mathbb{r}\leq y_1+\mathbb{r}$ for all $n\in\mathbb{N}$. It follows that
	$$
	|x_n|\leq|x_n-x|+|x|\leq y_1+\mathbb{r}+|x|
	$$
	for every $n$. Therefore, we obtain that $(x_n)$ is an order bounded sequence.
\end{proof}

\begin{proposition}\label{closedness}
	Let $E$ be a Riesz space and $\mathbb{c}$ be a linear convergence on $E$ having the lattice property. Then, the set $\mathcal{L}_{x_\alpha}^\mathbb{r}$ for an arbitrary net $(x_\alpha)_{\alpha\in A}$ in $E$ is $\mathbb{c}$-closed.
\end{proposition}

\begin{proof}
	Assume that $(z_\beta)_{\beta\in B}$ is a net in $\mathcal{L}_{x_\alpha}^\mathbb{r}$ which $\mathbb{c}$-converges to $z\in E$. We show that $z\in\mathcal{L}_{x_\alpha}^\mathbb{r}$ hold. For a fixed $\beta\in B$, we have $z_\beta\in \mathcal{L}_{x_\alpha}^\mathbb{r}$,  and so $x_\alpha\rc z_\beta$ as $\alpha\to\infty$. Thus, there exists a net $(w^\beta_\alpha)_{\alpha\in A}\cc\theta$ as $\alpha\to\infty$ such that $|x_\alpha-z_\beta|\leq w^\beta_\alpha+\mathbb{r}$ for all $\alpha\in A$. It follows that 
	$$
	|x_\alpha-z|\leq|x_\alpha-z_\beta|+|z_\beta-z|\leq w^\beta_\alpha+\mathbb{r}+|z_\beta-z|.
	$$
	holds for each $\alpha\in A$. On the other hand it follows from the linearity and the lattice property of $\mathbb{c}$-convergence that we obtain 
	$$
	(w^\beta_\alpha+|z_\beta-z|)_{(\alpha,\beta)\in A\times B}\cc \theta.
	$$
	Therefore, we get $x_\alpha\rc z$, and so we get $z\in \mathcal{L}_{x_\alpha}^\mathbb{r}$.
\end{proof}

\begin{theorem}
	Let $E$ be a Riesz space and $(x_\alpha)_{\alpha\in A}$ be a net in $E$. Then, the following statements hold:
	\begin{enumerate}
		\item[(i)] For any $w_1\in\mathcal{L}^{\mathbb{r}_1}_{x_\alpha}$ and $w_2\in\mathcal{L}^{\mathbb{r}_2}_{x_\alpha}$, $\gamma w_1+(1-\gamma)w_2\in \mathcal{L}^{\gamma\mathbb{r}_1+(1-\gamma)\mathbb{r}_2}_{x_\alpha}$ satisfies for each $0\leq\gamma\leq1$;
		\item[(ii)] $\mathcal{L}^{\mathbb{r}}_{x_\alpha}$ is a convex set.
	\end{enumerate}
\end{theorem}

\begin{proof}
	$(i)$	Since $w_1\in\mathcal{L}^{\mathbb{r}_1}_{x_\alpha}$ and $w_2\in\mathcal{L}^{\mathbb{r}_2}_{x_\alpha}$, there exist some nets $u_\alpha\cc\theta$ and $v_\alpha\cc\theta$ in $E$ such that 
	$$
	|x_\alpha-x|\leq u_\alpha+\mathbb{r}_1 \ \ \ \text{and} \ \ \ |x_\alpha-x|\leq v_\alpha+\mathbb{r}_2
	$$
	for each $\alpha\in A$. It follows that
	\begin{eqnarray*}
		\big|x_\alpha-\big(\gamma w_1+(1-\gamma)w_2\big)\big|&=& |x_\alpha-\gamma w_1+\gamma x_\alpha-\gamma x_\alpha+(1-\gamma) w_2|\\&\leq&\gamma|x_\alpha-w_1|+(1-\gamma)|x_\alpha-w_2|\\&\leq&\gamma u_\alpha+(1-\gamma)v_\alpha+\gamma\mathbb{r}_1+(1-\gamma)\mathbb{r}_2
	\end{eqnarray*}
	holds for all $\lambda\in A$ and $\gamma\in[0,1]$. By applying the linearity of $\mathbb{c}$, we get $(\gamma u_\alpha+(1-\gamma)v_\alpha)\cc \theta$, i.e., we obtain the desired result $\gamma w_1+(1-\gamma)w_2\in \mathcal{L}^{\gamma\mathbb{r}_1+(1-\gamma)\mathbb{r}_2}_{x_\alpha}$.
	
	$(ii)$ In the case $\mathbb{r}_1=\mathbb{r}=\mathbb{r}_2$, it follows from $(i)$ that $\mathcal{L}^{\mathbb{r}}_{x_\alpha}$ is a convex.
\end{proof}			

\begin{remark}
	Let $E$ be a Riesz space and $B$ be an order bounded subset of $E$. Then there exists a positive vector $e$ such that $\lvert x\rvert\leq e$ for all $x\in B$. Take a positive vector $\mathbb{r}\geq2e$. Then, for an arbitrary net $(x_\alpha)_{\alpha\in A}$ and a net $y_\alpha\cc \theta$, and any element $x$ in $B$, we have the following inequality;
	$$
	|x_\alpha-x|\leq|x_\alpha|+|x|\leq e+e\leq\mathbb{r}<\mathbb{r}+y_\alpha
	$$
	for all $\alpha\in A$. Therefore, we obtain $x\in\mathcal{L}^{\mathbb{r}}_{x_\alpha}$. Therefore, ever net in an order bounded set is rough $\mathbb{c}$-convergent to each element of the order bounded set.
\end{remark}

In the following result, we give a relation between rough $\mathbb{c}$-convergence and order convergence.
\begin{theorem}\label{order convergence}
	Let $E$ be a Riesz space, and $(x_{\alpha})_{\alpha\in A}$ be a net and $x$ an element in $E$. If $x_\alpha\co x$ hold, then $x\in \mathcal{L}^{\mathbb{r}}_{x_\alpha}$ for all $\mathbb{r}\geq\theta$ whenever $\mathbb{c}$ is order convergence.
\end{theorem}

\begin{proof}
	Assume that $x_\alpha\co x$ holds. Then, there exists a net $y_\alpha\downarrow \theta$ in $E$ satisfying $|x_\alpha-x|\le y_\alpha$ for all $\alpha\in A$. Therefore, for each positive vector $\mathbb{r}$, we have  $|x_\alpha-x|\le y_\alpha+\mathbb{r}$ for every $\alpha\in A$. So, we get $x\in \mathcal{L}^{\mathbb{r}}_{x_\alpha}$ whenever we take the $\mathbb{c}$-convergence as order convergence.
\end{proof}

\begin{theorem}
	Let $E$ be a Riesz space, and $(x_\alpha)_{\alpha\in A}$ and $(w_\alpha)_{\alpha\in A}$ be two nets in $E$. If $\mathbb{c}$ is a linear convergence and $|x_\alpha-y_\alpha|\cc \theta$ holds, then we have $\mathcal{L}^{\mathbb{r}}_{x_\alpha}=\mathcal{L}^{\mathbb{r}}_{w_\alpha}$.
\end{theorem}

\begin{proof}
	Let $x\in \mathcal{L}^{\mathbb{r}}_{x_\alpha}$ be an arbitrary element. Then, there exists a net $y_\alpha\cc \theta$ in $E$ satisfying $|x_\alpha-x|\le y_\alpha$ for all $\alpha\in A$. It follows that
	$$
	|w_\alpha-x|\leq |w_\alpha-x_\alpha|+|x_\alpha-x|\leq |w_\alpha-x_\alpha|+y_\alpha+\mathbb{r}
	$$
	satisfies for every $\alpha\in A$. By using the linearity of $\mathbb{c}$-convergence, we obtain $(|w_\alpha-x_\alpha|+y_\alpha)\cc \theta$, and so we have $x\in \mathcal{L}^{\mathbb{r}}_{w_\alpha}$. The same steps can be used to demonstrate the reverse inclusion. Hence, it has been omitted.
\end{proof}

We conclude the paper with the following result, focusing on the image of rough $\mathbb{c}$-limit points under  positive $\mathbb{c}$-continuous operators (i.e., a positive operator between Riesz spaces sends positive elements to positive elements).
\begin{theorem}
	Suppose that $T:E\to F$ is an operator between Riesz spaces and $(x_\alpha)_{\alpha\in A}$ is a net in $E$. If $T$ is positive $\mathbb{c}$-continuous operator, then $Tx\in\mathcal{L}_{T(x_\alpha)}^{T\mathbb{r}}$ holds for any $x\in\mathcal{L}^{\mathbb{r}}_{x_\alpha}$.
\end{theorem}

\begin{proof}
	Take $x\in\mathcal{L}_{x_\alpha}^{\mathbb{r}}$. Then, there exists a net $y_\alpha\cc \theta$ in $E$ satisfying $|x_\alpha-x|\le y_\alpha$ for all $\alpha\in A$. In the case $T=\theta$, the desired result is obvious. Assume that $T$ is a nonzero operator. For every $y_\alpha\cc\theta$, by using the fact in \cite[p.12]{ABPO}, the linearity and positivity of $T$ implies
	$$
	\lvert Tx_\alpha-Tx\rvert \leq T(|x_\alpha-x|)\leq  T(\mathbb{r})+T(y_\alpha).
	$$
	Hence, we have $Tx\in\mathcal{L}_{T(x_\alpha)}^{T\mathbb{r}}$ because $y_\alpha\cc\theta$ implies $T(y_\alpha)\cc \theta$ due to the $\mathbb{c}$-continuity of $T$.
\end{proof}

\section{Conclusion}
This paper makes a significant contribution to our understanding of convergence by introducing and explaining rough convergence within Riesz spaces. It provides valuable insights into how different convergences relate to each other, their properties, and the impact they have on mathematical analysis. By exploring rough convergence, it opens doors for further exploration and practical applications in various mathematical fields. The concept of rough $\mathbb{c}$-convergence creates a crucial connection between rough convergence and the well-established principles in Riesz space theory. Investigating rough $\mathbb{c}$-limit points in Riesz spaces yields captivating findings, shedding light on the intricate nature and importance of these limits concerning sequences and nets. This exploration offers intriguing observations, revealing the nuanced complexity and significance of these limits in mathematical contexts.

\end{document}